\documentclass[12pt]{amsart}
\usepackage[utf8]{inputenc}

\title{Non-commutative Rank and Semi-stability\\ of Quiver Representations}
\author{Alana Huszar}

\usepackage{amsmath}								
\usepackage{amssymb}
\usepackage{amsthm}
\usepackage{amscd}
\usepackage{amsfonts}
\usepackage{tikz}
\usetikzlibrary{matrix}
\usetikzlibrary{cd}
\usetikzlibrary{babel}
\usepackage{algpseudocode}
\usepackage{algorithm}

\usepackage[pdftitle={Quivers and Non-Commuative Rank},pdfauthor={Alana Huszar},ocgcolorlinks,linkcolor=linkblue,citecolor=linkred,urlcolor=linkred]{hyperref}
\usepackage{color}
\definecolor{linkred}{rgb}{0.6,0,0}
\definecolor{linkblue}{rgb}{0,0,0.6}

\usepackage{fullpage}
\usepackage{enumerate}

\newcommand{\ds}{\displaystyle}
\newcommand{\Mat}{\operatorname{Mat}}

\newtheorem{theorem}{Theorem}[section]
\newtheorem{proposition}[theorem]{Proposition}
\newtheorem{lemma}[theorem]{Lemma}
\newtheorem{corr}[theorem]{Corollary}
\newtheorem{alg}[theorem]{Algorithm}
\theoremstyle{definition}

\newtheorem{definition}[theorem]{Definition}

\newcommand{\rk}{\operatorname{rk}}
\newcommand{\Span}{\operatorname{span}}
\newcommand{\Hom}{\operatorname{Hom}}

\newcommand{\ext}{\operatorname{ext}}
\newcommand{\Rep}{\operatorname{Rep}}
\newcommand{\ncext}{\operatorname{ncext}}
\newcommand{\nchom}{\operatorname{nchom}}
\newcommand{\dimvec}{\underline{\rm dim}}
\newcommand{\im}{\operatorname{Im}}

\newcommand{\A}{\mathcal{A}}
\newcommand{\SL}{\operatorname{SL}}
\newcommand{\GL}{\operatorname{GL}}

\newcommand{\F}{{\mathbb F}}
\newcommand{\Z}{{\mathbb Z}}
\newcommand{\Q}{{\mathbb Q}}

\newcommand{\ncrk}{\operatorname{ncrk}}
\newcommand{\spn}{\operatorname{Span}}
\newcommand{\Ext}{\operatorname{Ext}}
\newcommand{\End}{\operatorname{End}}
\newcommand{\fmap}[2]{f_{#1}^{#2}}

\begin{document}

\begin{abstract}
Fortin and Reutenauer defined the non-commutative rank for a matrix with entries that are linear functions. The  non-commutative rank is related to stability in invariant theory,
non-commutative arithmetic circuits, and Edmonds' problem. We will generalize the  non-commutative rank to the representation theory of quivers and define 
non-commutative Hom and Ext spaces. We will relate these new notions  to King's criterion for $\sigma$-stability
of quiver representations, and the general Hom and Ext spaces studied by Schofield. We discuss polynomial time algorithms that compute
the non-commutative Homs and Exts and find an optimal witness for the $\sigma$-semi-stability
of a quiver representation.

\end{abstract}

\maketitle

\section{Introduction}
Given $A_1,A_2,\dots,A_m$, $n\times n$ matrices over a field $\F$, and $x_1,x_2,\dots,x_m$, variables in the free skew-field defined by Cohn in \cite{Cohn95}, Fortin and Reutenauer \cite{FR04} defined the non-commutative rank $\ncrk(A(\underline{x}))$ as the rank of 
the matrix of linear functions $A(\underline{x})=x_1A_1+x_2A_2+\cdots+x_mA_m$ over the free skew field. The non-commutative rank $\ncrk(A(\underline{x}))$ is also equal to the maximal value of 
$$\frac{\rk(X_1\otimes A_1+X_2\otimes A_2+\cdots+X_n\otimes A_n)}{d}
$$
where $d$ is a positive integer and $X_1,X_2,\dots,X_n$ are $d\times d$ matrices. A third characterization of non-commutative rank is in terms of shrunk subspaces.
Non-commutative rank is related to the notion of stability in geometric invariant theory. Consider the action of $\SL_n\times \SL_n$ on the space $\Mat_{n,n}^m$ of $m$-tuples of $n\times n$ matrices by left-right multiplication.
Then $\underline{A}=(A_1,A_2,\dots,A_m)$ is semi-stable with respect to this action if and only if $A(\underline{x})$ has full non-commutative rank, i.e., $\ncrk(A(\underline{x}))$ is equal to the maximal value $n$.

The example of $m$-tuples of $n\times n$ matrices fits in the broader framework of representation theory of quivers. A quiver $Q$ is just a directed graph,
with vertex set $Q_0$ and arrow set $Q_1$.
If we consider the generalized Kronecker quiver $Q$ with two vertices, say $x$ and $y$, and $m$ arrows from $x$ to $y$
and we choose the dimension vector $\alpha=(n,n)$, then the representation space $\Rep_\alpha(Q)$ of $\alpha$-dimensional representations of $Q$ is equal to $\Mat_{n,n}^m$. To construct a moduli space of $\alpha$-dimensional representations one needs to quotient out the $\SL_n\times \SL_n$ action. The goal of this paper is to generalize the notion of non-commutative rank and its properties to arbitrary quivers and dimension vectors. Geometric invariant theory for quiver representations was studied by King in \cite{King94}. For every quiver $Q$, dimension vector $\alpha\in \Z_{\geq 0}^{Q_0}$ and weight $\sigma\in \Z^{Q_0}$
King constructed a quotient for $\alpha$-dimensional representations of $Q$ with respect to the weight $\sigma$.
So for every weight $\sigma$ there is a notion of semi-stability for quiver representations and King gave a criterion  $\sigma$-semi-stability of a representation of~$Q$. 

In this paper, we connect non-commutative rank to $\sigma$-semi-stability of quiver representations. Through King's Criterion \cite{King94}, we discuss the importance of special subrepresentations that are optimal in witnessing the $\sigma$-semi-stability. We provide a framework to use existing algorithms to find these optimal $\sigma$-witnesses, as well as provide an algorithm using a sequence of subrepresentations. We then generalize work of Schofield on pairs of general representations, general ext, and general hom to pairs with one representation fixed, non-commutative ext, and non-commutative hom. We conclude by using non-commutative rank methods to demonstrate useful inequalities for the non-commutative ext. 

The Edmonds' problem, posed in 1967, asks to determine the rank of the $n\times n$ matrix $A(\underline{x})$, with homogeneous linear polynomials in $\Z[x_1,\ldots,x_m]$ over $\Q(x_1,\ldots,x_m)$ \cite{Edm67}. The decision version of this question, asking whether $A(\underline{x})$ has full rank or not, is known as the symbolic determinant identity testing problem (SDIT). We are considering instead the rank of $A(\underline{x})$ over the free skew field, and so the question of finding $\ncrk(A)$ is the \emph{non-commutative Edmonds' problem}, and the relaxation in simply deciding whether $A(\underline{x})$ has full non-commutative rank is the \emph{non-commutative full rank problem} (NCFullRank). Letting $\A=\Span\{A_1,A_2,\ldots,A_n\}$, we alternatively denote the non-commutative rank of $A(\underline{x})$ by $\ncrk(\A)$. Ivanyos, Qiao, and Subrahmanyam give equivalent formulations and history of NCFullRank in \cite{IQS17}. We are interested in the $c$-shrunk subspace, tensor blow-up, and particularly the nullcone formulations, which are discussed in Section~\ref{ncrkdefs}.

Lots of work from different angles has been done on this non-commutative rank. Cohn and Reutenauer proved NCFullRank was in PSPACE (can be solved using polynomial space) \cite{CR99}. Fortin and Reutenauer connected non-commutative rank explicitly to $c$-shrunk subspaces \cite{FR04}. Coming from studying non-commutative arithmetic circuits with divisions, Hrubes and Wigderson proved that non-commutative rank was equivalent to rank for large enough tensor blow-ups \cite{HW14}. Garg, Gurvitz, Oliveira, and Wigderson provide a polynomial time algorithm of non-commutative rank for fields of characteristic zero. In \cite{IKQS15}, for certain matrix spaces, Karpinski, Ivanyos, Subrahmanyam, and Qiao use Wong sequences to calculate the non-commutative rank. Building on this using blow-ups, the latter three authors provide an algorithm for finding the non-commutative rank of any matrix space \cite{IQS17}. Utilizing results on bounds from \cite{DM18b}, in \cite{IQS18}, they give a deterministic polynomial time algorithm.  

This problem can be expanded to that of finding a subrepresentation of a quiver representation $W$ that demonstrates the representation's semi-stability. This \emph{optimal $\sigma$-witness}, $W'$, can be found using algorithms finding a $c$-shrunk subsapce of a certain matrix space. Work by Chindris and Kline connect this problem to that of simultaneous robust subspace recovery (SRSR) \cite{CK20a}, and provide an algorithm for certifying the semi-stability of $W$ \cite{CK21B}.

In \cite{Scho92}, Schofield explored the dimension of $\Hom(V,W)$ and $\Ext(V,W)$ for a fixed quiver representation $V$, and generic quivers representation $W$ and $V$. This work was generalized by Crawley-Boevey, who described the asymptotic behavior of these dimensions when one of $W$ or $V$ is fixed (rather than both generic) \cite{Craw96}. We re-prove many of these results using insights from non-commutative rank methods, ultimately leading us to a bound on the asymptotic behavior.

\section{Non-commutative Rank}\label{ncrkdefs}
We will be concerned with the free skew field, made up of non-commuting polynomials, $\F\langle x_1,\ldots,x_m\rangle$, their inverses, and then enlarged to contain all sums, products, and inverses. The free skew field was first defined by Amitsur, \cite{Amit66}. In the free skew field, there is no standardized way to express elements, and elements may need to be defined with nested inverses. For example, $(x+yz^{-1}w)^{-1}$ can not be written without a nested inverse \cite{HW14}. 

Given a matrix, $A(\underline{x})$, with homogeneous linear polynomials in $\F\langle x_1,\ldots,x_m\rangle$, the non-commutative analogue of the Edmonds' problem asks to determine the rank of $A(\underline{x})$ over the free skew field. We denote this rank by $\ncrk(A)$. Similarly, the NCFullRank problem asks whether $A(\underline{x})$ has full rank over the free skew field. 
For example, we row reduce the following skew symmetric matrix over the free skew field to get:
\begin{equation}\label{OurExample}
T=\begin{bmatrix}
0 & x_1 & x_2\\ 
-x_1 & 0 &x_3 \\ 
-x_2 & -x_3  & 0
\end{bmatrix} \sim \begin{bmatrix}
0 & x_1 & x_2\\ 
-x_1 & 0 &x_3 \\ 
0 & 0  & x_3x_1^{-1}x_2-x_2x_1^{-1}x_3
\end{bmatrix}.
\end{equation}
Unfortunately, by the nature of the free skew field, it is hard to determine polynomial identities --- so it is not immediately clear if this matrix has non-commutative rank $2$ or $3$. For this reason, we explore additional equivalent formulations of non-commutative rank. For sketches on their equivalence, see \cite{IQS17}. 

We note here that many aspects of rank carry over to the non-commutative rank, for instance, the non-commutative row rank and column rank of $A(\underline{x})$ equal the $\ncrk(A)$ and we must have a minor with full rank equal to $\ncrk(A)$. We must still be careful, as other aspects do not: naively finding the ``determinant'' of $A(\underline{x})$, and comparing it to zero will not tell us whether the non-commutative rank is full (in fact, even how to define a single determinant in this context is unclear) \cite{GR91}.

\subsection{Blow-ups}
If $T=x_1A_1+\ldots x_mA_m$, let $\A=\spn{\{A_1,\ldots,A_m\}}$. The $d$th \emph{tensor blow-up} of $\A$ is $$\A^{\{d\}}:=M(d,\F)\otimes\A \subseteq M(dn,\F).$$ 
The rank of a matrix space, $\rk{\A}$, is the maximal $r$ so that there is a matrix with rank $r$ in $\A$. When $\F$ is large enough, $d$ divides the rank of $\A^{\{d\}}$  \cite{IQS17}. We have $$\ncrk(A)=\lim_{d\rightarrow\infty}\frac{\rk{\A^{\{d\}}}}{d}.$$

We may also write $\ncrk(\A)$ instead of $\ncrk(A)$. The value of $(\rk{\A^{\{d\}}})/d$ is increasing as $d$ increases, and is bounded by $n$. Derksen and Makam proved that if $\A$ has maximal non-commutative rank, then taking $d\geq n-1$ ensures $\rk{\A^{\{d\}}}=nd$ \cite{DM18}. 
If $\ncrk(A)=r<n$, then restricting to a full rank $r\times r$ submatrix of $A(\underline{x})$, we see
that $\rk{\A^{\{d\}}}=nd$ for $d\geq r-1$.
So we always have  $\rk{\A^{\{d\}}}=nd$ for $d\geq n-1$.

For our example~(\ref{OurExample}), take $d=2$. We then look for $2\times 2$ matrices $D_1,D_2,D_3$, so that $A_1\otimes D_1+A_2\otimes D_2+A_3\otimes D_3$ has max rank. Letting 
$$D_1=\begin{bmatrix}
1 & 0\\ 
0 & 0
\end{bmatrix}\!,\; D_2=\begin{bmatrix}
0 & 0\\ 
0 & 1 \\ 
\end{bmatrix}\!,\; D_3=\begin{bmatrix}
0 & 1 \\ 
1 & 0 \\ 
\end{bmatrix},$$

We find $$\rk{\left[
\begin{array}{c|c|c}
0 & D_1 & D_2\\
\hline
-D_1 & 0 & D_3 \\
\hline
-D_2 & -D_3 & 0
\end{array}
\right]}=6,$$
which must be maximal, and so $\ncrk(T)=3$.

\subsection{$c$-shrunk subspaces}
A subspace $U\subseteq \F^n$ is a \emph{$c$-shrunk subspace} of $\A$ if there exists a subspace $W\subseteq \F^n$ with $\dim(W)\leq\dim(U)-c$, and for every $A$ in $\A$, $A(U)\subseteq W$. The NCFullRank problem is equivalent to determining whether $\A$ has no $c$-shrunk subspace for $c>0$~\cite{Cohn95}. More generally \cite{FR04}, 
$$\ncrk(\A)=n-\max\{c \mid \text{there is a $c$-shrunk subspace of $\A$}\}.$$

Throughout the rest of this paper, we let $c=n-\ncrk(\A)$, i.e. all $c$-shrunk subspaces discussed are so that $c$ is maximal. 

\begin{lemma}\label{combinecshrunk}  Let $c=n-\ncrk(\A)$. If $U_1,U_2$ are $c$-shrunk subspaces of $\A$, then so are $U_1\cap U_2$ and $U_1+U_2$.
\end{lemma}
\begin{proof} By assumption $\dim U_i-\dim \A(U_i)=c$. Let $U_3=U_1\cap U_2$, $U_4=U_1+U_2$. We then have:
\begin{multline*}
c+c\geq ( \dim U_3-\dim \A(U_3))+(\dim U_4-\dim \A(U_4))=\\ = ( \dim(U_1\cap U_2)+\dim(U_1+U_2)) -(\dim(\A(U_1)\cap \A(U_2))+\dim(\A(U_1)+\A(U_2)))\geq\\ \geq
(\dim U_1+\dim U_2)-(\dim \A(U_1)+\dim \A(U_2))=\\=(\dim U_1-\dim \A(U_1))+(\dim U_2-\dim \A(U_2))=c+c.
\end{multline*}
We conclude that  $\dim U_3-\dim \A(U_3)=\dim U_4-\dim \A(U_4)=c,$ as $c$ is maximal.
Therefore, $U_3$ and $U_4$ are $c$-shrunk subspaces.
\end{proof}

In particular, there is a unique $c$-shrunk subspace of the lowest dimension, namely, the intersection of all $c$-shrunk subspaces. A recent similar discussion can be found in \cite{IMQ21}. In our skew-symmetric example~(\ref{OurExample}), although any matrix in $\A$ has rank $2$, the image of any subspace $U$ of $\F^3$ has the same dimension as $U$. In this case $c=n-\ncrk(\A)=3-3=0$, and the minimal $c$-shrunk subspace is the zero subspace.

\subsection{Semi-stability of Kronecker quiver}
A quiver $Q$ is a directed graph, with vertex set denoted $Q_0$ and arrow set denoted $Q_1$. A representation $W$ of a quiver $Q$, is an assignment of finite dimensional $\F$ vector spaces $W(x)$ to each $x$ in $Q_0$, and an assignment of linear maps $W(a)$ to each $a$ in $Q_1$. We let $ha$ and $ta$ denote the head and tail vertices of the arrow $a$ respectively. A dimension vector is a function $\alpha:Q_0\to {\mathbb N}=\{0,1,2,\dots\}$.
The dimension vector $\underline{\dim} W$
of a representation $W$ is defined by $(\underline{\dim} W)(x)=\dim W(x)$.
Fixing $Q$ and a dimension vector $\alpha$, there is an action on quiver representations by $\GL(\alpha):=\prod_{x\in Q_0} \GL(\alpha(x))$. The action of $(Y(x),x\in Q_0)$ takes $W(a)$ to $Y(ha)W(a)Y(ta)^{-1}$ for all $a\in Q_1$, and leaves each $W(x)$ with $x\in Q_0$ unchanged. For a path $p=a_ja_{j-1}\cdots a_1$, we denote by $W(p)$ the composition of linear maps $W(a_j)W(a_{j-1})\cdots W(a_1)$. The empty path from vertex $x$ to itself is denoted by $e_x$ and $W(e_x)$ is defined as the identity map of $W(x)$.

The representations of $Q$ with dimension vector $\alpha$ (indexed by the vertices) is denoted $\Rep_\alpha(Q)$. A representation is semi-stable if its orbit closure does not contain the zero representation; representations that are not semi-stable define the \emph{nullcone}. No acyclic quiver representations are semi-simple. Instead, for a weight $\sigma$ in $\Z^{Q_0}$, we additionally use the $1$-dimensional representation $\chi_\sigma$, a \emph{character} with action of $\GL(\alpha)$ given by multiplication by
$$\chi_\sigma(Y(x),x\in Q_0)=\prod_{x\in Q_0}\det(Y(x))^{\sigma(x)}.$$
A representation $W$ is $\sigma$ semi-stable if $(W,1)$ is semi-stable in $\Rep_\alpha(Q)\oplus \chi_\sigma$.

The NCFullRank problem for $T=x_1A_1+\ldots,x_mA_m$ is equivalent to determining whether the quiver representation $W$,

\begin{center}
\begin{tikzcd}
\F^n \arrow[r, draw=none, "\raisebox{+1.5ex}\vdots" description] \arrow[r,shift left=1.7ex,"A_1"]  \arrow[r,shift right=1.2ex,swap,"A_m"] & \F^n
\end{tikzcd}\end{center}

is $\sigma$-semistable, for $\sigma=(1,-1)$. In our skew-symmetric matrix example~(\ref{OurExample}), we would like to determine whether the above quiver with
$$A_1=\begin{bmatrix}
0 & 1 & 0\\ 
-1 & 0 & 0\\ 
0 & 0 & 0
\end{bmatrix}\!,\; A_2=\begin{bmatrix}
0 & 0 & 1\\ 
0 & 0 & 0\\ 
-1 & 0 & 0
\end{bmatrix}\!,\; A_3=\begin{bmatrix}
0 & 0 & 0\\ 
0 & 0 & 1\\ 
0 & -1 & 0
\end{bmatrix}$$
is $(1,-1)$-semi-stable.

We would like to be able to relate quiver representations to the non-commutative rank, rather than just to NCFullRank. To do this, we need a way of measuring how far a representation $V$, is from being $\sigma$-semistable. For this, we use King's Criterion \cite{King94}. For a representation $W$, let $\sigma(\dimvec(W))=\sum \dim(W(x))\sigma(x)$.

\begin{proposition}[King's Criterion, \cite{King94}]
A representation $W$ in $\Rep_\alpha(Q)$ is $\sigma$-semi-stable if and only if $\sigma(\alpha)=0$ and $\sigma(\dimvec(W))\leq 0$ for all subrepresentations $W'$ of $W$.
\end{proposition}

\begin{proposition}\label{cshrunksubrep}
    Given $\A=\spn{\{A_1,\ldots,A_m\}}$, and $c$ the maximum $\sigma(\dimvec(W'))$ over all subrepresentations $W'$ of the Kronecker quiver with maps $\{A_i\}$, the $\ncrk(\A)=n-c$.
\end{proposition} 

\begin{proof}
If $W$ is a Kronecker quiver, and $\sigma=(1,-1)$, let $W'$ be a subrepresentation with $c:=\sigma(\dimvec(W'))$ maximal. Then, since $\A(W'(x))$ is contained in $W'(y)$, $W'(x)$ is a $c$-shrunk subspace. On the other hand, if instead we start with a $c$-shrunk subspace $U$,  $W'(x):=U$, and $W'(y):=\sum_{i=1}^{m}A_i\, U$. This defines a subrepresentation $W'$, where $\sigma(\dimvec(W'))=c$. So for Kronecker quivers, $c$-shrunk subspaces give us subrepresentations $W'$ with $\sigma(\dimvec(W'))$ maximal, and vice-versa. 
\end{proof}

So, the non-commutative rank of $\A$ is equal to the maximum of $\sigma(\dimvec(W'))$ over all subrepresentations $W'$ of the Kronecker quiver with maps $A_1,\ldots,A_n$.

\section{Reduction to Kronecker Quiver}\label{reduction}

Through Proposition \ref{cshrunksubrep}, if a representation $W$ is not $\sigma$-semi-stable, we can still measure its closeness to $\sigma$-semi-stability by finding a subrepresentation $W'$ with $\sigma(\dimvec(W'))$ maximal. In \cite{CK21B}, this is called the \emph{discrepancy} of $(W,\sigma)$. Note that we are now no longer limited to the Kronecker quiver --- we can now ask this question for any acyclic quiver $W$, for any~$\sigma$. We call a subrepresentation $W'$ which maximizes $c=\sigma(\dimvec(W'))$ an optimal $\sigma$-witness. When $\sigma$ is understood, we call this $W'$ an \emph{optimal witness}. We can generalize Lemma \ref{combinecshrunk} for subrepresentations.

\begin{proposition}
If $W_1,W_2$ are optimal $\sigma$-witnesses of $W$, then so are $W_1\cap W_2$ and $W_1+W_2$. In particular, there is a minimal and maximal optimal $\sigma$-witness.
\end{proposition}
\begin{proof}
Let $c$ be the discrepancy of $(W,\sigma)$. Let $\sigma_+(x)=\max\{0,\sigma(x)\}$, and similarly, $\sigma_-(x)=-\min\{0,\sigma(x)\}$. For $i=1,2$, by assumption $$\sigma(\dimvec(W_i))=\sum_{x\in Q_0} \big(\sigma_+(x)-\sigma_-(x)\big)\dim W_i(x)=c.$$ Let $W_3=W_1\cap W_2$, $W_4=W_1 + W_2$. We then have:
\begin{multline*}
c+c\geq \sum \big(\sigma_+(x)-\sigma_-(x)\big)\dim W_3(x)+\big(\sigma_+(x)-\sigma_-(x)\big)\dim W_4(x)=\\=
\sum \sigma_+(x)\dim W_3(x)-\sigma_-(x)\dim W_3(x)+\sigma_+(x)\dim W_4(x)-\sigma_-(x)\dim W_4(x)=\\=
\sigma_+(x)\big(\dim W_1(x)+\dim W_2(x)\big)-\sigma_-(x)\big(\dim W_1(x)+\dim W_2(x)\big)=c+c.
\end{multline*}
We conclude that  $\sigma(\dimvec(W_3))=\sigma(\dimvec(W_4))=c,$ as $c$ is maximal.
Therefore, $W_3$ and $W_4$ are optimal $\sigma$-witnesses. We can find a minimal optimal $\sigma$-witness by taking the intersection of all optimal $\sigma$-witnesses, and similarly find a maximal optimal $\sigma$-witness by taking the sum of all optimal $\sigma$-witnesses.
\end{proof}

We would like to extend the techniques in \cite{IQS17} in order to find an optimal $\sigma$-witness. To do this, we reduce any acyclic quiver to the Kronecker quiver. We use the construction described in \cite{DM18}, but provide an altered set up, using presentations as in \cite{DF15}. Let $P_x$ be the indecomposable representation of $Q$ with basis given by all paths starting at vertex ~$x$. 

Let $P_x$ be the indecomposable projective representation corresponding to vertex $x$. So, $P_x(y)=e_y\F Qe_x$, with basis given by paths from $x$ to $y$. Let $\mathbf{P_1}:=\bigoplus_{x\in Q_0}P_x\,^{\sigma_-(x)}$, and $\mathbf{P_0}:=\bigoplus_{x\in Q_0}P_x\,^{\sigma_+(x)}$. Consider all possible morphisms $\varphi$ between the quiver representations $$\varphi:\mathbf{P_1}\rightarrow\mathbf{P_0}.$$
To our above set of morphisms, apply $\Hom(\cdot,W)$ to get
$$A(\varphi):\Hom(\mathbf{P_0},W)\rightarrow\Hom(\mathbf{P_1},W),$$
where $A(\varphi):=\Hom(\varphi,W)$. We can consider a subspace $\Hom(P_x\,^{\sigma_+(x)},W')$ as $Z_+(x)\otimes W'(x)$ for some $Z_+(x)=\F\,^{\sigma_+(x)}$. Notice $\Hom(\mathbf{P_0},W)$ is a right $\End(\mathbf{P_0})$-module by precomposition. Let $x,y$ be so that both $\sigma_+(x)$ and $\sigma_+(y)$ are positive. Note $\End(\mathbf{P_0})$ contains $H=\prod_{x\in Q_0}\GL(\sigma_+(x))$, a reductive group, which acts on the $Z_+(x)$, leaving the $W(x)$ alone. So, an $H$-subrepresentation of $\bigoplus_{x\in Q_0} Z_+(x)\otimes W(x)$ must be of the form $\bigoplus_{x\in Q_0}Z_+(x)\otimes W'(x)$ for some subspaces $W'(x)$ of each $W(x)$. Our set of maps can also be considered between the spaces 

$$A(\varphi):\bigoplus_{x\in Q_0} W(x)\,^{\sigma_+(x)}\rightarrow\bigoplus_{x\in Q_0}W(x)\,^{\sigma_-(x)}.$$

Now, we have a matrix space $\A$ consisting of all $A(\varphi)$. This is the space of block matrices with blocks mapping $W(x)$ to $W(y)$ given by a linear combination of $W(p)$, where $p$ is a path from $x$ to $y$. For this new Kronecker Quiver, we may run the algorithm in \cite{IQS17} to get the minimal $c-$shrunk subspace of $\bigoplus_{x\in Q_0} W(x)\,^{\sigma_+(x)}$, $U$. 

\begin{lemma}
The minimal $c$-shrunk subspace, $U\subseteq\Hom(\mathbf{P_0},W)$, is a left $\End(\mathbf{P_0})$ module, and $\sum_\varphi A(\varphi) U$ is a left $\End(\mathbf{P_1})$ module.
\end{lemma}\label{fixedlem}
\begin{proof}
First, we prove that given any $c$-shrunk subspace, $U$, and invertible $T$ in $\End(\mathbf{P_0})$, $T\cdot U$ is also $c$-shrunk. We have the image of $T \cdot U$:

$$
\sum_\varphi A(\varphi) (T \cdot U)=\sum_\varphi A(\varphi \cdot T)U =\sum_\varphi A(\varphi) U. 
$$
Here the sum is taken over all morphisms $\varphi$ as above. 
It follows that
$$
\dim \sum_\varphi A(\varphi) (T \cdot U)=\dim \sum_{\varphi} A(\varphi) U.
$$
As $T$ is an automorphism, we also have $\dim T \cdot U=\dim U$, so $T \cdot U$ is $c$-shrunk. If $U$ is the minimal $c$-shrunk subspace, $T \cdot U$ is also $c$-shrunk and of the same dimension, so $T \cdot U=U$. As $\End(\mathbf{P_0})$ is spanned by invertible elements, this shows that the minimal $c$-shrunk subspace $U$ is a left $\End(\mathbf{P_0})$ module. Similarly, given $S$ in $\End(\mathbf{P_1})$, we see that \[S\cdot\sum_\varphi A(\varphi) (U)=\sum_\varphi A(S\cdot\varphi)U =\sum_\varphi A(\varphi) U.\]
\end{proof}

\begin{theorem}
Given the minimal $c-$shrunk subspace for the set of linear maps
$$A(\varphi):\bigoplus_{x\in Q_0} W(x)\,^{\sigma_+(x)}\rightarrow\bigoplus_{x\in Q_0}W(x)\,^{\sigma_-(x)},$$
we can construct a subrepresentation of $W$, $W'$, so that $\sigma(\dimvec(W'))$ is maximal. Furthermore, $\sigma(\dimvec(W'))=c$. 
\end{theorem}

\begin{proof}
Considered as a subspace of $\bigoplus Z(x)\otimes W(x)$, the minimal $c$-shrunk $U$ is of the form $\bigoplus Z(x)\otimes W'(x)$, for some subspaces $W'(x)$ of $W(x)$. For $y$ so that $\sigma_+(y)=0$, define $\displaystyle W'(y)=\sum_{a:x\rightarrow y}W(a)W'(x)$. This ensures we have a subrepresentation. Note that $c\leq \sum\dim(W'(x)^{\sigma_+(x)})-\sum\dim(W'(x)^{\sigma_-(x)})$, but $c$ is maximal, so $\sigma(\dimvec(W'))=c$. We note that the $W'(y)$ are similarly closed under the action of $\End(\mathbf{P_1})$.

If there were a subrepresentation $W''$ with $\sigma(\dimvec(W''))$ less than $c$, Note that $U'=\displaystyle\bigoplus_{x\in Q_0} W''(x)^{\sigma_+(x)}$ is a shrunk subspace, with $\dim(U')-\dim(\A( U'))>c$, so $c$ would not be maximal.
\end{proof}

\subsection{Algorithms}

After using this reduction of a quiver representation to a Kronecker quiver, we can employ any previous algorithms or other techniques for finding a $c-$shrunk subspace. 
 If we successfully find a $c-$shrunk subspace, $U$, that is not minimal, we can construct a $c-$shrunk subspace that is fixed under the action of $\End(P_1)$ by taking instead $$\bigcap_{T\in \End(\mathbf{P_1})}T\cdot U.$$ Such a subspace will give a optimal $\sigma$ witness. We may use a basis of $\End(\mathbf{P_1})$ to get this subspace in polynomial time. 
 
 In \cite{IKQS15}, Wong sequences, originally defined by Kai-Tek Wong \cite{Wong74}, are used in certain cases to find a $c-$shrunk subspace. In \cite{IQS17}, blow-ups are used to extend the original algorithm to find a $c-$shrunk subspace for any collection of matrices. The algorithm takes a matrix $A$ in $\A$, constructing a sequence starting with $W_0=0$, and letting $W_{i+1}=\A A^{-1}(W_i)$. This sequence stabilizes to some subspace, $W^*$. In the case that $W^*$ is contained in $\im{A}$, $A^{-1}(W^*)$ is a $c-$shrunk subspace with $c$ maximal and equal to $n-\rk{A}$. In this case, where the algorithm returns a $c-$shrunk subspace, we claim the subspace is minimal.

The minimal shrunk subspace, $U$, is the intersection of all $c-$shrunk subspaces, so $U\subseteq A^{-1}(W^*)$. The limit of the sequence $W^*$ is the smallest subspace $Z$ so that $\displaystyle\bigcup_{i=1}^m A_i^{-1}(Z)$ contains $A^{-1}(Z)$. So by minimality, $U$ is returned when this sequence terminates with $W^*$ contained in $\im(A)$. In the case where blow-ups are invoked to find a $c-$shrunk subspace, the same sequence is used in the larger space, finding a $cd-$shrunk subspace. This by the same reasoning must be minimal, so when pulled back to a $c-$shrunk subspace in the original space, it must remain minimal.  

Let $n=\min\{\sum \sigma_+(x)\dim W(x),\sum \sigma_-(y)\dim W(y)\}$. For sufficiently large fields, ($|\F|>2n$) there is a randomized algorithm to find a $c-$shrunk subspace \cite[Corollary 1.5]{IQS17}. This randomized algorithm is much simpler and typically must faster than the deterministic algorithm. In the context of representations, this algorithm immediately after reduction, blow up by a sufficiently large \cite{IQS18,DM18} factor, $d\geq n-1$. In this blow-up, randomly choose a matrix
$$A:\bigoplus_{x\in Q_0} W(x)\,^{d\sigma_+(x)}\rightarrow\bigoplus_{x\in Q_0}W(y)\,^{d\sigma_-(y)},$$
where $A$ is in $\A^{\{d\}}:=M(d,\F)\otimes\A$.
Through the Schwarz-Zippel-DeMillo-Lipton lemma \cite{Schw80,Zip79,DL78}, if a field is large enough, evaluating a non-zero polynomial over that field at a randomly chosen point is likely to give a non-zero result. Taking the determinant of minors of a matrix in the blow-up, we are likely to have $\rk A= \ncrk \A^{\{d\}}$. Thus, running the Wong sequence on this $A$ will result in the return of a $cd-$shrunk subspace \cite[Lemma 9]{IKQS15}. From this $cd$-shrunk subspace in the blow-up, we can find a $c-$shrunk subspace of $\bigoplus_{x\in Q_0} W(x)\,^{\sigma_+(x)}$, constructing a subrepresentation as above.

The deterministic Wong sequence algorithm for finding non-commutative rank, introduced in \cite{IKQS15}, uses a sequence of subspaces, testing its limit, $W^*$ for evidence of a $c$-shrunk subspace. In the quiver representation context, we would like to instead use a sequence of subrepresentations. In this deterministic setting, we only need $|\F|>n$.

To do this, we again start with a random matrix $A$ in the blow-up, as above. Next, find a pseudo-inverse of $A$, a matrix $B$ so that $B$'s restriction to $\im(A)$ is the inverse to $A$'s restriction to a direct complement of $\ker(A)$. Note that $B$ is a block matrix as well, with blocks mapping each $W(y)$ for $\sigma(y)<0$ to each $W(x)$ with $\sigma(x)>0$. Let $I_x$ index the $|d\sigma(x)|$ copies of $W(x)$. Let $\pi_{x,i}:\bigoplus_{x\in Q_0} W(x)\,^{d\sigma_+(x)}\rightarrow W(x)$ be the projection to the $i$th copy of $W(x)$. Each projection can be thought of as coming from the action of $\End(\mathbf{P_0})$. Similarly, define this for vertices $y$ with $\sigma_{-}(y)>0$. 

For each block, take the projection $\pi_{y,i} B \pi_{x,j}$. This gives a linear map from $W(x)$ to $W(y)$. Construct a new quiver representation, $W^{+}$, on a new quiver $Q^{+}$ by adding arrows $p:y\rightarrow x$ for each block in the pseudo-inverse, with each $W^{+}(p)$ defined as $\pi_{y,i} B \pi_{x,j}$. 

Define a subspace at vertices $x$ with $\sigma(x)>0$ of $W^{+}$: 
$$K(x):=\sum_{i\in I_{x}} \pi_{x,i}\ker(A).$$
For all other vertices, define $K(y)=0.$ Let $W'$ be the smallest subrepresentation of $W^{+}$ containing each $K(x)$. Note that $W'$ must also be a subrepresentation of our original $W$. 

\begin{proposition}
For $W'$ as defined above, $\bigoplus_{x\in Q_0} W'(x)\,^{d\sigma_+(x)}$ is $cd$-shrunk, with image (under $\A^{[d]}$) $\bigoplus_{x\in Q_0} W'(x)\,^{d\sigma_-(x)}$. Thus, $W'$ is an optimal $\sigma$ witness.
\end{proposition}

First, we claim that $\bigoplus_{x\in Q_0} W'(x)\,^{d\sigma_+(x)}$ is the minimal $cd$-shrunk subspace of $\A^{\{d\}}$. By construction, the Wong sequence algorithm returns the smallest subspace containing $\ker(A)$, and closed under $\A^{\{d\}}$ and our pseudo-inverse $B$. The $K(x)$ must remain inside the minimal shrunk subspace, as the projections come from $\End(\mathbf{P_0})$. Similarly, the new maps in $W^{+}$ come from the action of $\End(\mathbf{P_0})\bigoplus\End(\mathbf{P_1})$, so in finding the smallest subrepresentation, we must still remain in the minimal shrunk subspace (at positive vertices). So in finding the minimal representation of $W^+$ that contains each $K(x)$, $W'$, we get the smallest subspace $\bigoplus_{x\in Q_0} W'(x)\,^{d\sigma_+(x)}$ containing $\ker(A)$ and closed under $\A^{\{d\}}$ and $B$, i.e. the minimal $cd$-shrunk subspace. 

\begin{proposition}
Given a quiver representation $W$, a weight vector $\sigma$, $|\F|>n$, letting $n_x:=\dim(W(x))$, and $N=\sum_{x\in Q_0} n_x$, there is an algorithm polynomial time in the $n_x$ to find an optimal $\sigma$ witness.
\end{proposition}

Recalling the above discussion, we first construct $Q^+$ and $W^+$. To do this, we chose a random matrix in the $d=N-1$ blowup, $A$, and find its pseudo-inverse, $B$, which takes polynomial time ($\leq(dN)^3$). We then construct new linear maps for each of the $d^2\sigma_+\sigma_-$ blocks in $B$ by composing $B$ with projection maps. This composition is matrix multiplication, which can be done in polynomial time. Next we contruct $K(x)$ at each vertex $x$, which is the sum over the $d\sigma(x)$ projections of $\ker(A)$. We can find a basis for $\ker(A)$ itself in polynomial time using row reduction. Last, we use Algorithm \ref{findwitness} to loop through all our arrows $N$ times, to find the optimal $\sigma$ witness, $W'$ from the $K(x)$. This algorithm will stabilize at the $N$th loop or shorter, as each iteration of the outside loop will either raise the dimension of the current $W'$, or will not (in which case, we are done, we have found the final $W'$). We can increase the dimension at most $N$ times, so this must be a correct bound for the number of times to run the outer loop. 

\begin{alg}\label{findwitness}
Algorithm for finding $W'$.\\
\emph{Input:} Quiver $Q$, Representation $W$ of $Q$, subspaces $K(x)\subseteq W(x)$ for all vertices $x$.\\
\emph{Output:} Smallest subrepresentation $W'$ so that $K(x)\subseteq W'(x)$ for all vertices $x$.
\begin{algorithmic}[1]
\State $W'(x) = K(x)$ for all $x$;
\For{$i=1$ \textbf{ to } $N$}
    \For{$a \in Q_1$}
       \State $W'(ha)= W'(ha)+W(a)W'(ta)$;
    \EndFor
\EndFor
\end{algorithmic}
\end{alg}

\section{Non-commutative General Ext and Hom}

In \cite{Scho92}, Schofield studied the minimal dimension of $\Ext(V,W)$ for representations $V$ and $W$ of dimension vectors $\alpha$ and $\beta$. Define
$$Z^t(\alpha,\beta):=\{(V,W)\in \Rep_\alpha(Q)\times\Rep_\beta(Q)|\dim\Hom_Q(V,W)\geq t\}.$$
Each of these subsets of $\Rep_\alpha(Q)\times\Rep_\beta(Q)$ are closed. Take $t$ the minimal positive value of $\dim\Hom_Q(V,W)$. Then, $Z^{t+1}(\alpha,\beta)$ is a proper closed subset. We call the pair $(V,W)$ $(\alpha,\beta)$-general if they are in the (open and dense) complement of $Z^{t+1}(\alpha,\beta)$. On this complement, $\dim\Hom_Q(V,W)$ is constant, as is $\dim\Ext_Q(V,W)$. Schofield calls these generic hom and ext respectively:
\begin{eqnarray*}
\hom(\alpha,\beta)&=&\dim(\Hom_Q(V,W)), \text{ and}\\
\ext(\alpha,\beta)&=&\dim(\Ext_Q(V,W)).
\end{eqnarray*}

Crawley-Boevey generalized the generic hom in \cite{Craw96}, which we show along with the generalization of generic ext. To do this, fix a representation $W$ of $Q$, with dimension vector $\beta$. Define now 
$$Z^t(\alpha,W):=\{V\in \Rep_\alpha(Q)|\dim\Hom_Q(V,W)\geq t\}.$$
Again, each of these subsets are closed, and we take $t$ minimal so that $\dim\Hom_Q(V,W)$ is positive. We call $V$ $(\alpha,W)$-general if it is in the complement of $Z^{t+1}(\alpha,W)$.  

\begin{definition}
Let $V$ be an $(\alpha,W)$-general representation. We define the $(\alpha,W)$-general hom and ext as
\begin{eqnarray*}
\hom(\alpha,W)&=&\dim(\Hom_Q(V,W)), \text{ and}\\
\ext(\alpha,W)&=&\dim(\Ext_Q(V,W)).
\end{eqnarray*}
\end{definition}

\begin{lemma} We have 
$$
\ext(\alpha,W)\geq \max\{-\langle \alpha, \dimvec W'\rangle \mid \mbox{$W'$ factor representation of $W$}\}.
$$
\end{lemma}
\begin{proof}
If $V$ is a general representation of dimension $\alpha$, then applying $\Hom_Q(V,\cdot)$ to 
$$0\to W''\to W\to W'\to 0$$
gives
an exact sequence
$$
\cdots\to \Ext_Q(V,W)\to \Ext_Q(V,W')\to 0,
$$
so $\dim \Ext_Q(V,W')\leq \dim \Ext_Q(V,W)$ and $\ext(\alpha,W')\leq \ext(\alpha,W)$.
We get
$$
-\langle \alpha,\dimvec W'\rangle=\ext(\alpha,W')-\hom(\alpha,W')\leq \ext(\alpha,W')\leq \ext(\alpha,W).
$$
\end{proof}

\begin{definition}\label{nchomdef} The non-commutative ext and hom are defined by the following limits of ext and hom:
\begin{eqnarray*}
\ncext(\alpha,W)&=&\lim_{d\to\infty}\frac{\ext(d\alpha,W)}{d}\\
\nchom(\alpha,W)&=&\lim_{d\to\infty}\frac{\hom(d\alpha,W)}{d}.
\end{eqnarray*}
\end{definition}
Note that for every representation $W$ of dimension $\beta$, we have $\nchom(\alpha,W)-\ncext(\alpha,W)$ equal to $\langle\alpha,\beta\rangle$. These limits were originally studied in \cite{Craw96}, though we give them a name to highlight their connection to non-commutative rank, as seen in the next discussion and proposition.

We have a map $$\fmap{W}{\alpha}:\Rep_\alpha \longrightarrow\Hom\Big(\bigoplus_{x\in Q_0}\Hom\big(\F^{\alpha(x)},W(x)\big)\rightarrow\bigoplus_{a\in Q_1}\Hom\big(\F^{\alpha(ta)},W(ha)\big) \Big)$$
given by sending a representation $V$ to the map $\fmap{W}{\alpha}(V)$, which takes the set of $\varphi(x)$ from $\Hom(\F^{\alpha(x)},W(x))$ over all vertices $x$ to the set of maps $\varphi(ha)V(a)-W(a)\varphi(ta)$ over all arrows $a$. Note that the kernel of each $\fmap{W}{\alpha}(V)$ is $\Hom_Q(V,W)$, and and the cokernel is $\Ext_Q(V,W)$. From this point forward, we will refer to the image of $\fmap{W}{\alpha}$ (the set of $\fmap{W}{\alpha}(V)$ over all $V$), as simply $\fmap{W}{\alpha}$ itself. 

Next, we note that we can consider $\Rep_{d\alpha}$ as the blow-up of $\Rep_\alpha$ as follows. Each $Z$ in $\Rep_{d\alpha}$ is so that $Z(x)\cong \F^{\alpha(x)}\otimes U(x)$ for $U(x)\cong\F^d$. At the arrows, we have $Z(a)\cong\sum V_i(a)\otimes U_i(a)$, a finite sum where each $U_i(a)$ is a $d\times d$ matrix over $\F$, and each $V_i$ is from $\Rep_\alpha$. Now, given a $\overline{V}$ in $\Rep_{d\alpha}$, we get a map:
$$\Hom\Big(\bigoplus_{x\in Q_0}\Hom\big(\F^{d\alpha(x)},W(x)\big)\xrightarrow{\fmap{W}{d\alpha}(\overline{V})}\bigoplus_{a\in Q_1}\Hom\big(\F^{d\alpha(ta)},W(ha)\big) \Big).$$
Notice that we can find $\ncrk{(\fmap{W}{\alpha})}$ using $\ncrk{(\fmap{W}{d\alpha})}$ and dividing by $d$ since $\fmap{W}{d\alpha}$ is the $d$th blow-up of $\fmap{W}{\alpha}$. 

\begin{proposition}\label{samed}
The rank and non-commutative rank of $\fmap{W}{d\alpha}$ are equal if and only if $\ds\nchom(\alpha,W)=\frac{\hom(d\alpha,W)}{d}.$
\end{proposition}

For a $(d\alpha,W)$-general $\overline{V}$ in $\Rep_{d\alpha}$, the kernel of $\fmap{W}{d\alpha}(\overline{V})$ is of minimal dimension. So, $\rk(\fmap{W}{d\alpha})=\sum d\alpha(x)\beta(x)-\hom(d\alpha,W)$. We get $$\frac{\rk \fmap{W}{d\alpha}}{d}=\sum\alpha(x)\beta(x)-\frac{\hom(d\alpha,W)}{d},$$
showing that $d$ which maximizes the left-side (giving us the non-commutative rank), maximizes the right side (minimizing $\frac{\hom(d\alpha,W)}{d}$, giving us the non-commutative hom).

\begin{corr}
Given dimension vector $\alpha$, and a representation $W$ of dimension $\beta$, the $d$ in the limit of definition \ref{nchomdef} can be chosen to be $$\min\Big\{\sum_{x\in Q_0}\alpha(x)\beta(x)-1,\sum_{a\in Q_1}\alpha(ta)\beta(ha)-1\Big\}$$.
\end{corr}
Recall the bound for non-commutative rank blow-ups from \cite{DM18} is $n-1$, where $n$ is the dimension of both the domain and co-domain. We may not have a space of square matrices, so a large enough $d$ will be found when we first reach either $\sum_{x\in Q_0}\alpha(x)\beta(x)-1$ or $\sum_{a\in Q_1}\alpha(ta)\beta(ha)-1$. 

\begin{theorem}\label{ncext}
We have
$$
\ncext(\alpha,W)=\max\{-\langle\alpha,\dimvec W'' \rangle \mid \mbox{$W''$ factor representation of $W$}\}.
$$
\end{theorem}
\begin{proof}

Choose $d$ so that $\ncext(d\alpha,W)$ equals $\frac{\ext(d\alpha,W)}{d}$. Look at the set of maps:
$$\Hom\Big(\bigoplus_{x\in Q_0}\Hom\big(\F^{\alpha(x)},W(x)\big)\xrightarrow{\fmap{W}{d\alpha}(\overline{V})}\bigoplus_{a\in Q_1}\Hom\big(\F^{\alpha(ta)},W(ha)\big) \Big),$$
for all representations $\overline{V}$ in $\Rep_{d\alpha}$. By Proposition \ref{samed}, this set of maps has non-commutative rank equal to its rank. So we can find the minimal $c$-shrunk subspace, which:
\begin{enumerate}
    \item has the form $\bigoplus_{x\in Q_0}\Hom\big(\F^{d\alpha(x)},W'(x)\big)$, for some subrepresentation $W'$ of $W$ (from discussion in section \ref{reduction}), and
    \item has image of the form $\bigoplus_{a \in Q_1}\Hom\big(\F^{d\alpha(ta)},W'(ha)\big)$.
\end{enumerate}
So we get $c=d\sum\alpha(x)\dim(W'(x))-d\sum\alpha(ta)\dim(W'(ha))=\langle d\alpha,\dimvec(W')\rangle$, but $c$ is the non-commutative rank, so also can be found by $\sum d\alpha(x)\beta(x)-\rk(\fmap{W}{d\alpha})=\hom(d\alpha,W).$ This leaves us with $\frac{\hom(d\alpha,W)}{d}=\langle\alpha,\dimvec(W')\rangle$ after dividing by $d$.
As for non-commutative ext, we then get $\ncext(\alpha,W)=\nchom(\alpha,W)-\langle\alpha,\beta\rangle$, finally leaving us with $\ncext(\alpha,W)=-\langle\alpha,\dimvec W'' \rangle$, for $W''=W/W'.$

\end{proof}

We note that we can dually fix a representation $V$, and look at $\hom(V,\beta)$ and $\ext(V,\beta)$ to define $\nchom(V,\beta)$ and $\ncext(V,\beta)$. 
\begin{definition} The non-commutative ext and hom are defined by the following limits of ext and hom:
\begin{eqnarray*}
\ncext(V,\beta)&=&\lim_{d\to\infty}\frac{\ext(V,d\beta)}{d}\\
\nchom(V,\beta)&=&\lim_{d\to\infty}\frac{\hom(V,d\beta)}{d}
\end{eqnarray*}
\end{definition}
\begin{theorem}
We have
$$
\ncext(V,\beta)=\max\{-\langle \dimvec V', \beta \rangle \mid \mbox{$V'$ subrepresentation of $V$}\}.
$$
\begin{proof}
The proof follows from duality of theorem \ref{ncext}. We note that this can also be seen by using Corollary 1 from \cite{Craw96}, by subtracting $\langle \dimvec V, \beta \rangle$.
\end{proof}
\end{theorem}

\begin{corr}
For large enough $|\F|$, there are both deterministic and randomized algorithms for calculating $\ncext(\alpha,W),\nchom(\alpha,W),\ncext(V,\beta),$ and $\nchom(V,\beta)$.
\begin{proof}
We can apply any of the algorithms used to find $c$-shrunk subspaces to the set of maps $\fmap{W}{\alpha}({V})$ or $\fmap{\beta}{V}(W)$ respectively, and use the dimension of $c$ to calculate the non-commutative ext and hom.
\end{proof}
\end{corr}

\section{Acknowledgements}
The author was supported by NSF grant DGE 1256260, and would additionally like to thank Harm Derksen for direction and conversation, and Calin Chindris and Daniel Kline for comments on an earlier draft of this paper.

\bibliography{bb}{}
\bibliographystyle{alpha}
\end{document}